\documentclass[12pt]{amsart}
\usepackage{amssymb,slashed, color,array}
\usepackage{amsmath,amsfonts,amsthm,euscript,mathrsfs,multirow}
\usepackage[all]{xy}
\usepackage[english]{babel}
\usepackage{epsf}
\usepackage{graphicx}
\usepackage{bbm}
\csname @addtoreset\endcsname{equation}{section}

\newtheorem{theorem}{Theorem}[section]
\newtheorem{lemma}[theorem]{Lemma}

\theoremstyle{definition}
\newtheorem{definition}[theorem]{Definition}
\newtheorem{remark}{Remark}[section]

\newenvironment{notation and conventions}{\textbf{Notation and conventions.}}{ }
\DeclareFontFamily{U}{rsf}{} \DeclareFontShape{U}{rsf}{m}{n}{ <5> <6> rsfs5 <7> <8> <9> rsfs7 <10-> rsfs10}{}
\DeclareMathAlphabet\Scr{U}{rsf}{m}{n}
\definecolor{pink}{rgb}{1,0,1}

  \usepackage[pdftex]{hyperref}

\begin{document}

\begin{center}
\baselineskip=14pt{\LARGE
 On tadpole relations via Verdier specialization\\
}
\vspace{1.5cm}
{\large  James Fullwood
  } \\
\vspace{.6 cm}

Institute of Mathematical Research, The University of Hong Kong, Pok Fu Lam Road, Hong Kong.\\

\end{center}

\vspace{1cm}
\begin{center}

{\bf Abstract}
\vspace{.3 cm}
\end{center}

{\small
Using the construct of `Verdier specialization', we provide a purely mathematical derivation of Chern class identities which upon integration yield the D3-brane tadpole relations coming from the equivalence between F-theory and associated weakly coupled type IIB  orientifold limits. In particular, we find that all Chern class identities associated with weak coupling limits appearing in the physics literature are manifestations of a relative version of Verdier's specialization formula.}

\tableofcontents{}






\section{Introduction}\label{intro}
Let $\varphi:Y\to B$ be an elliptic fibration over a smooth complete complex algebraic variety of arbitrary dimension whose total space $Y$ is a smooth hypersurface in a $\mathbb{P}^2$-bundle $\pi:\mathbb{P}(\mathscr{E})\to B$ given by a Weierstra{\ss} equation
\[
Y:\left(y^2=x^3+fxz^2+gz^3\right)\subset \mathbb{P}(\mathscr{E}).
\]
The coefficients $f$ and $g$ are then sections of line bundles over $B$ so that the fiber $\varphi^{-1}(b)$ over $b\in B$ is given by a Weierstra{\ss} equation with coefficients $f(b)$ and $g(b)$. The fibers of $\varphi$ will degenerate to singular cubics over the discriminant hypersurface
\[
\Delta:(4f^3+27g^2=0)\subset B,
\] 
with a generic fiber over $\Delta$ being a nodal cubic which will degenerate further to a cuspidal cubic over the codimension two locus
\[
C:(f=g=0)\subset B,
\]
which in fact coincides with the singular locus of $\Delta$ in the case that the differentials $df$ and $dg$ are everywhere linearly independent. 

In \cite{AE1}, string-theoretic arguments (i.e., S-duality between F-theory and associated weakly-coupled type IIB orientifold limits) led to the discovery of an interesting identity in the Chow group $A_*B$ between the Chern-Schwartz-MacPherson (or simply CSM) class\footnote{We review the theory of CSM classes of constructible functions in \S\ref{CSMV}.} of the constructible function $\mathbbm{1}_{\Delta}+\mathbbm{1}_C$ and the CSM class of a constructible function which seems to have no relation to $\mathbbm{1}_{\Delta}+\mathbbm{1}_C$, namely
\begin{equation}\label{MI}
c_{\text{SM}}\left(\mathbbm{1}_{\Delta}+\mathbbm{1}_C\right)=c_{\text{SM}}\left(2\mathbbm{1}_O+\mathbbm{1}_D-\mathbbm{1}_S\right),
\end{equation}
where the varieties $O$, $D$ and $S$ arise when viewing $Y$ as a smooth deformation (parametrized by a disk in $\mathbb{C}$) of a certain singular variety $Y_0$. In the case that $Y$ is a Calabi-Yau fourfold, the degree zero component of both sides of equation \ref{MI} yield a numerical relation predicted by string dualities which reflect the equivalence of D3-brane charge between F-theory and its weakly-coupled orientifold limit. Such relations equating charges between dual theories are often referred to in the physics literature as `tadpole relations'. 

In subsequent works \cite{AE2}\cite{CCG}\cite{EFY}\cite{EJY}, similar identities associated with other elliptic fibrations not in Weierstra{\ss} form were also arrived at via physical considerations. While the identities were indeed shown to hold in all these cases, a precise mathematical explanation as to why such identities should exist was lacking. In Remark~4.5 of \cite{AE1}, the construct known as `Verdier specialization' --which associates constructible functions on a family over a disk with constructible functions on its central fiber-- was used \emph{a posteriori} to sketch a derivation of identity \ref{MI} solely from mathematical principles, but in an indirect manner that would have been unclear if the identity was not known to already hold in the first place. In any case, it was this example which motivated us to employ Verdier specialization to give a top-down approach to the explanation of the appearance of such identities. As such, we find all Chern class identities appearing in \cite{AE1}\cite{AE2}\cite{CCG}\cite{EFY}\cite{EJY} to be manifestations of the following
\begin{theorem}\label{VT}
Let $\varphi:Y\to B$ be an elliptic fibration whose total space $Y$ is a smooth zero-scheme of a section of a vector bundle $\mathscr{V}\to \mathbb{P}(\mathscr{E})$, and let $s_t$ be a smoothly varying family of sections of $\mathscr{V}\to \mathbb{P}(\mathscr{E})$  over an open disk $\mathscr{D}$ about the origin in $\mathbb{C}$ whose zero-schemes give rise to a family $\mathscr{Y}\to \mathscr{D}$, such that the total space $\mathscr{Y}$ is smooth over $\mathscr{D}\setminus\{0\}$. Denote the central fiber of the family by $\varphi_0:Y_0\to B$. Then
\begin{equation}\label{mt}
c_{\emph{SM}}\left(\varphi_*\mathbbm{1}_{Y}\right)=c_{\emph{SM}}\left({\varphi_0}_*\sigma \mathbbm{1}_{\mathscr{Y}}\right),
\end{equation}
where $\varphi_*\mathbbm{1}_{Y}$, ${\varphi_0}_*\sigma \mathbbm{1}_{\mathscr{Y}}$ are the pushforwards to $B$ of the characteristic function $\mathbbm{1}_{Y}$ and Verdier's specialization $\sigma \mathbbm{1}_{\mathscr{Y}}$ of the characteristic function of the family $\mathscr{Y}$.
\end{theorem}
Theorem 1.1 is essentially a relative version of Verdier's formula \ref{VF}, the proof of which we provide in \S\ref{CSMV}. When $Y$ is an elliptic fibration in Weierstra{\ss} form as defined earlier, computing both sides of equation \ref{mt} as given in Theorem~\ref{VT} yields precisely identity \ref{MI} (as we show in \S\ref{ls}). In what follows we review the theory of CSM classes of constructible functions along with the specialization morphisms first presented by Verdier, prove Theorem~\ref{VT}, discuss the physical motivation behind identities such as \ref{MI}, and then show by explicit computation how Theorem~\ref{VT} yields such Chern class identities.

\section{CSM classes, Verdier specialization, and proof of Theorem~\ref{VT}}\label{CSMV} Let $X$ be a complex variety. A constructible function on $X$ is an integer-valued function of the form
\[
\sum_i a_i\mathbbm{1}_{W_i},
\]
with each $a_i\in \mathbb{Z}$, $W_i\subset X$ a closed subvariety and $\mathbbm{1}_{W_i}$ the function that evaluates to $1$ for points inside of $W_i$ and is zero elsewhere. The collection of all such functions forms an abelian group under addition, and is referred to as the \emph{group of constructible functions} on $X$, denoted $F(X)$. A proper morphism $f:X\to Y$ induces a functorial group homomorphism $f_*:F(X)\to F(Y)$, which by linearity is determined by the prescription
\begin{equation}\label{cfd}
f_*\mathbbm{1}_W(p)=\chi\left(f^{-1}(p)\cap W\right),
\end{equation}
where $W\subset X$ is a closed subvariety and $\chi$ denotes topological Euler characteristic with compact support. By taking $F(f)=f_*$, we may view $F$ as a covariant functor from varieties to abelian groups. Another covariant functor from varieties to abelian groups is the homology functor $H_*$, which takes a variety to its integral homology. Motivated by a conjecture of Deligne and Grothendieck, in 1974 MacPherson explicitly constructed a natural transformation
\[
c_*:F\to H_*,
\]
such that for $X$ smooth
\[
c_*(\mathbbm{1}_X)=c(TM)\cap [X]\in H_*X, 
\] 
i.e., the total homological Chern class of $X$ \cite{RMCC}. The class $c_*(\mathbbm{1}_X)$ for arbitrary $X$ is then a functorial generalization of Chern class to the realm of singular varieties. Moreover, such a class provides a means of generalizing the Gau{\ss}-Bonnet theorem to the singular setting, as functoriality implies
\[
\int_X c_*(\mathbbm{1}_X)=\chi(X),
\]
where the integral sign denotes proper pushforward to a point. As the class $c_*(\mathbbm{1}_X)$ was later shown by Brasselet and Schwartz to coincide with the Alexander-dual of a class constructed by Schwartz in the 1960s, we now refer to it as the \emph{Chern-Schwartz-MacPherson \emph{(}or simply CSM\emph{)} class}. Moreover, for an arbitrary constructible function $\delta\in F(X)$ we will refer to $c_*(\delta)$ as the `CSM class of $\delta$', which will be denoted from here on by $c_{\text{SM}}(\delta)$.

Now suppose $\mathcal{Z}\to \mathscr{D}$ is a family over an open disk about the origin in $\mathbb{C}$ such that it is topologically locally trivial over $\mathscr{D} \setminus \{0\}$, denote a general fiber over $\mathscr{D}\setminus \{0\}$ by $Z_t$ and denote the central fiber by $Z_0$. In \cite{VS}, Verdier defines specialization morphisms
\[
\sigma_H:H_*(Z_t)\to H_*(Z_0), \quad \quad \sigma_F:F(\mathcal{Z})\to F(Z_0), 
\]
and proves that for any constructible function $\vartheta \in F(\mathcal{Z})$
\begin{equation}\label{VF}
\sigma_Hc_{\text{SM}}(\left.\vartheta\right|_{Z_t})=c_{\text{SM}}(\sigma_F(\vartheta))
\end{equation}
for $t$ sufficiently small.

The only instance of this formula which will concern us is in the context of Theorem~\ref{VT}, i.e., when $\mathcal{Z}=\mathscr{Y}$ and $\vartheta=\mathbbm{1}_{\mathscr{Y}}$ as given in the statement of Theorem~\ref{VT}. In such a case $\mathscr{Y}$ is the family corresponding to a smooth deformation of a fibration $Y_0$ whose total space is singular, so that $\left.\vartheta\right|_{Z_t}=\mathbbm{1}_Y$. Moreover, as observed for example in the proof of Theorem~5.3 in \cite{PPCC} we have 
\[
\sigma_Hc_{\text{SM}}(\mathbbm{1}_Y)=c_{\text{F}}(Y_0), 
\]
where $c_{\text{F}}(Y_0)$ denotes the \emph{Chern-Fulton class} of $Y_0$. We won't give the full definition of Chern-Fulton class, but for varieties admitting a virtual tangent bundle their Chern-Fulton class coincides with the total Chern class of its virtual tangent bundle capped with its fundamental class, so
\begin{equation}\label{FC}
c_{\text{F}}(Y_0)=c(TY_0^{\text{vir}})\cap [Y_0],
\end{equation}
where $TY_0^{\text{vir}}$ denotes the virtual tangent bundle of $Y_0$ (i.e., $T\mathbb{P}(\mathscr{E})-N_{Y_0}\mathbb{P}(\mathscr{E})$). Thus formula \ref{VF} in the context of Theorem~\ref{VT} yields
\begin{equation}\label{f1}
c_{\text{F}}(Y_0)=c_{\text{SM}}(\sigma_F(\mathbbm{1}_{\mathscr{Y}})).
\end{equation}

As formula \ref{f1} is sufficient for proving Theorem~\ref{VT}, we hold off giving a precise definition of $\sigma_F$ until \S\ref{ls}, where we actually need to explicitly compute $\sigma_F(\mathbbm{1}_{\mathscr{Y}})$. Moreover, since any reference to $\sigma_H$ will no longer be needed we will denote $\sigma_F$ simply by $\sigma$ from here on. We are now in position to prove Theorem~\ref{VT}, but before doing so, we recall our precise set-up.

Let $B$ be a smooth complete complex variety of arbitrary dimension endowed with a vector bundle $\mathscr{E}\to B$, let $\pi:\mathbb{P}(\mathscr{E})\to B$ denote the projective bundle of \emph{lines} in $\mathscr{E}$ and let $\varphi:Y\to B$ be a proper surjective morphism such that $Y$ is smooth and the generic fiber of $\varphi$ is an elliptic curve. We further assume that the total space $Y$ is the zero-scheme associated with a section of a vector bundle $\mathscr{V}\to \mathbb{P}(\mathscr{E})$, and that $\varphi=\pi\circ i_Y$, where $i_Y:Y\hookrightarrow \mathbb{P}(\mathscr{E})$ is the natural inclusion. Now suppose $\mathscr{D}$ is an open disk about the origin in $\mathbb{C}$ and let $s_t$ denote a smoothly varying family of sections of $\mathscr{V}\to \mathbb{P}(\mathscr{E})$ over $\mathscr{D}$. Then by considering the zero-schemes of the sections $s_t$ we obtain a family $\mathscr{Y}\to \mathscr{D}$ whose fiber $Y_t$ over $t\in \mathscr{D}$ comes naturally equipped with a morphism $\varphi_t:Y_t\to B$ (by composing the natural inclusion with the bundle projection $\pi$). Our last assumption is that there exists a $t_0\in D\setminus \{0\}$ such that $Y_{t_0}=Y$ and that the total space of the family $\mathscr{Y}\to D$ is smooth away from the central fiber $Y_0$ \footnote{Though Verdier makes the assumption of topological local triviality over $D\setminus 0$ in order to employ his specialization morphisms, we use a generalization of Verdier's specialization morphisms given by Aluffi\cite{VSA}, which only requires that the family to which you are applying the specialization morphisms to be smooth away from its central fiber, and this is the context in which we will work throughout.}. We recall that $\sigma \mathbbm{1}_{\mathscr{Y}}$ is a contructible function on $Y_0$, thus we may push it forward to $B$ via ${\varphi_0}_*$, and similarly we may pushforward $\mathbbm{1}_Y$ to $B$ via $\varphi_*$. The conclusion of Theorem~\ref{VT} then states
\[
c_{\text{SM}}\left(\varphi_*\mathbbm{1}_{Y}\right)=c_{\text{SM}}\left({\varphi_0}_*\sigma \mathbbm{1}_{\mathscr{Y}}\right),
\]  
which we now prove.

\begin{proof}[Proof of Theorem~\ref{VT}]
Let $i_{Y_0}:Y_0\hookrightarrow \mathbb{P}(\mathscr{E})$ denote the natural inclusion, so that $\varphi_0=\pi\circ i_{Y_0}$. Then
\begin{eqnarray*}
{i_{Y_0}}_*c_{\text{SM}}\left(\sigma \mathbbm{1}_{\mathscr{Y}}\right)&\overset{\ref{f1}}=&{i_{Y_0}}_*c_{\text{F}}(Y_0) \\
&\overset{\ref{FC}}=&{i_{Y_0}}_*\left(\frac{c\left(i_{Y_0}^*T\mathbb{P}(\mathscr{E})\right)}{c\left(i_{Y_0}^*\mathscr{V}\right)}\cap [Y_0] \right)\\
&=&\frac{c\left(T\mathbb{P}(\mathscr{E})\right)}{c\left(\mathscr{V}\right)}\cap {i_{Y_0}}_*[Y_0] \\
&=&\frac{c\left(T\mathbb{P}(\mathscr{E})\right)}{c\left(\mathscr{V}\right)}\cap {i_Y}_*[Y] \\
&=&{i_Y}_*\left(\frac{c\left(i_{Y}^*T\mathbb{P}(\mathscr{E})\right)}{c\left(i_{Y}^*\mathscr{V}\right)}\cap [Y] \right) \\
&=&{i_Y}_*c(Y),
\end{eqnarray*}
where have repeatedly used adjunction, the projection formula, and the fact that $Y$ and $Y_0$ are both zero-schemes of sections of $\mathscr{V}\to \mathbb{P}(\mathscr{E})$ (e.g. their normal bundles are both restrictions of $\mathscr{V}$). Thus
\begin{eqnarray*}
c_{\text{SM}}\left(\varphi_*\mathbbm{1}_Y\right)&=&\varphi_*c_{\text{SM}}\left(\mathbbm{1}_Y\right) \\
&=&\varphi_*c(Y) \\
&=&\pi_*\circ {i_Y}_* c(Y) \\
&=&\pi_*\circ {i_{Y_0}}_*c_{\text{SM}}\left(\sigma \mathbbm{1}_{\mathscr{Y}}\right) \\
&=&{\varphi_0}_*c_{\text{SM}}\left(\sigma \mathbbm{1}_{\mathscr{Y}}\right) \\
&=&c_{\text{SM}}\left({\varphi_0}_*\sigma \mathbbm{1}_{\mathscr{Y}}\right), \\
\end{eqnarray*} 
where we just repeatedly use functoriality of $c_{\text{SM}}$ and the fourth equality follows from our first string of equalities, concluding the proof.
\end{proof}

As mentioned in \S\ref{intro}, Theorem~\ref{VT} is essentially just a relative version of formula \ref{VF} in the category of $B$-schemes. Moreover, the assumption that $Y$ was in fact an elliptic fibration was unnecessary, but we didn't want to stray too far out of context. We want to reiterate that the main contribution of this note however is not Theorem~\ref{VT}, but the realization that Chern class identities appearing in the physics literature on weak coupling limits of F-theory are merely different instances of this simple relation.

\begin{remark}
Mac Pherson's natural tansformation $c_*$ has since been extended by Kennedy to the case where $\mathbb{C}$ is replaced by an arbitrary algebraically closed field of characteristic zero and $H_*$ is replaced by the Chow functor $A_*$ \cite{KCC}. The Chow functor $A_*$ is covariant with respect to proper morphisms, takes a variety to its group of algebraic cycles modulo rational equivalence, and acts on proper morphisms via proper pushforward of algebraic cycles (see \cite{IntersectionTheory}, \S~1.4.). In light of this, we may work in the context of Chow homology as opposed to integral homology, and this is our preference moving forward.  
\end{remark} 

\section{Tadpole relations between F-theory and type IIB}\label{tr}
The geometric apparatus of a regime of string theory referred to as `F-theory' is an elliptic fibration $Y\to B$ (usually with a section) with $Y$ a smooth Calabi-Yau fourfold (over $\mathbb{C}$) playing the r\^{o}le of the compactified dimensions of spacetime. The theory may be smoothly deformed to weakly coupled type IIB orientifold theories via certain deformations (parametrized by a disk in $\mathbb{C}$) of the fibration to singular fibrations of the form $Y_0\to B$ where all the fibers are singular degenerations of elliptic curves. Such deformations are referred to as `weak coupling limits' (or `Sen limits'), as they correspond to diminishing the string coupling constant $g_s$ to a value sufficiently close to zero (as needed for perturbative expansions). The spacetime of the type IIB theory however is not $Y_0$, but a certain Calabi-Yau double cover $X\to B$ (ramified over a smooth divisor) naturally associated with $Y_0$. In the physics literature the Calabi-Yau variety $X$ is referred to as an \emph{orientifold} and the ramification divisor is referred to as the \emph{orientifold \emph{(}or O7-\emph{)} plane}. Such a passage from a strongly coupled theory to a physically equivalent weakly coupled theory is a manifestation of what physicists call `S-duality'.  As D3-branes are S-duality invariant, they do not depend on the string coupling constant, thus if F-theory and its weakly coupled counterpart are indeed equivalent descriptions of nature then the induced D3 charge in both theories should be equal. An equation relating the charges in both theories is what we refer to as a `tadpole relation'.

In F-theory the induced D3 charge is proportional to the topological Euler characteristic of the total space $Y$ of the elliptic fibration, whereas in type IIB the induced D3 charge (in the absence of fluxes) come from divisors in the type IIB spacetime $X$ which support the orientifold plane and D7-branes (which in the case of a weak coupling limit arise as pullbacks to $X$ of components of the flat limit of the corresponding family of discriminants). The charges of the orientifold plane and D7-branes were initially thought also to be proportional to the Euler characteristic of the varieties on which they are supported, but in certain weak coupling limits of F-theory a single D7-brane will wrap a singular variety whose local equations coincide with that of a Whitney umbrella, and it was shown in \cite{AE1}\cite{AE2}\cite{CDM} that a correction factor must be added in order to account for the singularities to obtain the proper charge. For example, for the case of an elliptic fibration $\varphi:Y\to B$ in Weierstra{\ss} form as introduced in \S\ref{intro}, the tadpole relation shown to hold in \cite{AE1} is
\begin{equation}\label{tr1}
\chi(Y)=2\chi(O)+\chi(D)-\chi(S),
\end{equation}
where the LHS of \ref{tr1} represents the D3 charge in F-theory and the right hand side represents the D3 charge in its type IIB limit. The orientifold plane is denoted here by $O$, the D7 brane (which has Whitney umbrella singularities in this case) is denoted by $D$ and $S$ denotes the pinch locus of $D$, so that the correction factor in this case is $-\chi(S)$. Moreover, it was shown in \cite{AE1} that equation \ref{tr1} holds at the level of Chern classes after pushing forward to $B$ (and without any Calabi-Yau or dimensional hypothesis), yielding the Chern class identity 
\[
c_{\text{SM}}\left(\mathbbm{1}_{\Delta}+\mathbbm{1}_C\right)=c_{\text{SM}}\left(2\mathbbm{1}_O+\mathbbm{1}_D-\mathbbm{1}_S\right),
\]
i.e., equation \ref{MI}. Similar Chern class identities were then shown to hold in an analogous manner in the works \cite{AE2}\cite{CCG}\cite{EFY}\cite{EJY}, i.e., by starting with the tadpole relations predicted by physics and then showing that the relations hold at the level of Chern classes. What we give here is a top-down approach to the existence of such identities, as we show that such Chern class identities arising from weak coupling limits of F-theory are particular instances of Theorem~\ref{VT}, i.e., the relation
\begin{equation}\label{mf1}
c_{\text{SM}}\left(\varphi_*\mathbbm{1}_{Y}\right)=c_{\text{SM}}\left({\varphi_0}_*\sigma \mathbbm{1}_{\mathscr{Y}}\right),
\end{equation}
where the specialization $\sigma:F(Y)\to F(Y_0)$ is from constructible functions on the F-theory elliptic fibration $Y$ to constructible functions on its singular degeneration $Y_0$ corresponding to a weak coupling limit. Furthermore, when tadpole relations involve singular branes the appropriate correction factor accounting for the singularities automatically appears on the RHS of equation \ref{mf1}. As an illustration, in \S\ref{ls} we show this is the case by explicitly computing both sides of equation \ref{mf1} for the case of a Weiertra{\ss} fibration and its weak coupling limit first constructed by Sen \cite{Sen}. In \S\ref{AA}, we give the precise definitions of all non-Weiertra{\ss} fibrations appearing in the works \cite{AE2}\cite{CCG}\cite{EFY}\cite{EJY}, along with their weak coupling limits and associated Chern class identities corresponding to tadpole relations. We omit the derivation of the associated Chern class identities via Theorem~\ref{VT} for these non-Weierstra{\ss} fibrations, as the computations follow as in the Weierstra{\ss} case \emph{mutatis mutandis}.

\section{Verification of the Weierstra{\ss} case}\label{ls} Let $B$ be a smooth complete complex variety endowed with an ample line bundle $\mathscr{L}\to B$. As defined in \S\ref{intro}, the Weierstra{\ss} fibration (also often referred to as an `$E_8$' fibration in the physics literature) is given by the equation  
\[
Y:\left(y^2=x^3+fxz^2+gz^3\right)\subset \mathbb{P}(\mathscr{E}),
\]
where\footnote{Here and in the cases of all other fibrations presented in \S\ref{AA} the total space of the elliptic fibration will be Calabi-Yau if and only if $\mathscr{L}=\mathscr{O}(-K_B)$. Thus the ampleness assumption on $\mathscr{L}$ then requires $B$ to be Fano in the Calabi-Yau case.} $\mathscr{E}=\mathscr{O}_B\oplus \mathscr{L}^2\oplus \mathscr{L}^3$. The coefficients $f$ and $g$ are then general sections of $\mathscr{L}^4$ and $\mathscr{L}^6$ respectively and $x$, $y$ and $z$ are chosen to be sections of appropriate line bundles so that $Y$ ends up being a smooth zero-scheme associated with a section of the line bundle $\mathscr{O}_{\mathbb{P}(\mathscr{E})}(3)\otimes \mathscr{L}^6$. We recall that the discriminant of $Y$ is given by 
\[
\Delta:\left(4f^3+27g^2=0\right)\subset B.
\]
In the case that $df$ and $dg$ are everywhere linearly independent $\Delta$ is singular along
\[
C:\left(f=g=0\right)\subset B.
\]
The singular degeneration of $Y$ corresponding to the weak coupling limit first constructed by Sen \cite{Sen} corresponds to setting
\[
f=-3h^2+c\eta, \quad g=-2h^3+ch\eta+c^2\chi,
\]
where $h$, $\eta$ and $\chi$ are appropriate sections of tensor powers of $\mathscr{L}$ and $c$ is our deformation parameter. Varying $c$ over a disk $\mathscr{D}$ about the origin in $\mathbb{C}$ then gives rise to a family $\mathscr{Y}\to \mathscr{D}$, whose total space is smooth smooth away from its central fiber, which is given by
\[
Y_0:\left(y^2=x^3-3h^2xz^2-2h^3z^3\right)\subset \mathbb{P}(\mathscr{E}).
\] 
The fibers of $Y_0$ over the complement of $\{h=0\}$ are all nodal cubics while the fibers over $\{h=0\}$ degenerate further to cuspidal cubics. The flat limit as $c\to 0$ of the corresponding family of discriminants is given by 
\[
\Delta_0:\left(h^2(\eta^2+12h\chi)=0\right)\subset B.
\]
The varieties then contributing to the D3 charge on the type IIB side are then the scheme-theoretic pullbacks of the irreducible components of $\Delta_0$ to 
\begin{equation}\label{of}
X:\left(z_o^2=h\right)\subset \mathscr{L},
\end{equation}
which is a double cover of $B$ ramified over $\{h=0\}$. It is straightforward to show that $Y$ is Calabi-Yau if and only if $X$ is, and in this case $X$ then plays the r\^{o}le of the compactified dimensions of spacetime in an `orientifolded' type IIB theory. In any case, by projecting to $B$ we may compute the $D3$ charges on the type IIB side in terms of the varieties
\[
O:(h=0)\subset B, \quad D:(\eta^2+12h\chi=0)\subset B, \quad S:(h=\eta=\chi=0)\subset B,
\]
i.e., the irreducible components of the limiting discriminant $\Delta_0$ and the singular locus of $D$ which we denote by $S$ (which  according to \cite{CDM} (\S3.4.4) must be incorporated into the charge in order to account for the singularities of $D$). We assume here that $h$ is general so that $O$ is smooth. As mentioned in the previous section, the presumed tadpole relation predicted by F-theory/type IIB duality is then
\[
\chi(Y)=2\chi(O)+\chi(D)-\chi(S),
\]
which may be shown to hold by integrating both sides of the following Chern class identity
\begin{equation}\label{f2}
c_{\text{SM}}(\mathbbm{1}_{\Delta}+\mathbbm{1}_{C})=c_{\text{SM}}\left(2\mathbbm{1}_O+\mathbbm{1}_D-\mathbbm{1}_S\right),
\end{equation}
which was verified to hold by explicit computation of both sides in \cite{AE1}. We now show that identity \ref{f2} which was first arrived at by physical considerations is an explicit representation of the equation
\begin{equation}\label{f3}
c_{\text{SM}}\left(\varphi_*\mathbbm{1}_{Y}\right)=c_{\text{SM}}\left({\varphi_0}_*\sigma \mathbbm{1}_{\mathscr{Y}}\right),
\end{equation}
as given by Theorem~\ref{VT}. 

\begin{remark}A key insight of Clingher, Donagi and Wijnholt in \cite{SL} was that the varieties $O$, $D$ and the seemingly auxiliary type IIB orientifold $X$ all naturally arise after blowing up the singular locus of the total space of the family $\mathscr{Y}$. Thus one does not have to compute limiting discriminants and then pullback to $X$ to arrive at the corresponding D-brane configuration on the type IIB side. Moreover, it is precisely this blowup that will be needed in order to compute $\sigma \mathbbm{1}_{\mathscr{Y}}$ appearing in the RHS of equation \ref{f3} \footnote{A calculation of the degree-zero component RHS of \ref{f3} in the case where $Y$ is a Calabi-Yau fourfold in Weierstra{\ss} form was given in \cite{SL}.}. (we will give the definition of $\sigma$ following Lemma~\ref{l1}). Upon doing so we will point out how $O$, $D$ and $X$ naturally arise via the blowup as first observed in \cite{SL}. 
\end{remark}

Now both sides of equation \ref{f3} involve pushing forward constructible functions, for which we will employ the following 
\begin{lemma}\label{l1}
Let $f:Z\to V$ be a proper morphism of varieties and let $\{U_i\}$ be a stratification of $V$ with $U_i$ locally closed such that the fibers are topologically constant on each $U_i$. Denote by $F_i$ the fiber over $U_i$ and write $U_i=V_i\setminus W_i$ with $V_i$ and $W_i$ closed subvarieties of $V$ for each $i$. Then
\[
f_*\mathbbm{1}_{Z}=\sum_i \chi(F_i)\left(\mathbbm{1}_{V_i}-\mathbbm{1}_{W_i}\right).
\] 
\end{lemma}

We omit the proof of Lemma~\ref{l1} as it follows directly from the definition of $f_*$ as given by equation \ref{cfd}. The last ingredient we need in order to compute the RHS of equation \ref{f3} is the definition of Verdier's sepcialization morphism from constructible functions on $Y$ to constructible functions on $Y_0$ which we denote here by $\sigma$, so that we may compute $\sigma \mathbbm{1}_{\mathscr{Y}}$. While Verdier's original definition of $\sigma$ for general families was given in the analytic category, we give an equivalent formulation $\sigma \mathbbm{1}_{\mathscr{Y}}$ first given by Aluffi \cite{VSA}, which holds over any algebraically closed field of characteristic zero (and applies to more general families than the case considered by Verdier, though we give the definition in terms of the context at hand).
\begin{definition}\label{SD} 
Let $\mathcal{Z}\to \mathscr{D}$ be a family over a disk about the origin in $\mathbb{C}$ such that the total space $\mathcal{Z}$ is smooth over $\mathscr{D}\setminus \{0\}$, denote its central fiber by $Z_0$, and let $\psi:\widetilde{\mathcal{Z}}\to \mathcal{Z}$ be a proper birational morphism such that $\widetilde{\mathcal{Z}}$ is smooth, $\mathcal{D}=\psi^{-1}(Z_0)$ is a divisor with normal crossings with smooth components, and $\psi$ restricted to the complement of $\mathcal{D}$ is an isomorphism (such a $\psi$ exists by resolution of singularities). Let $\delta$ be the constructible function on $\mathcal{D}$ given by 
\[
\delta(p)=\begin{cases} m \quad  \text{if $p$ lies on a single component of $\mathcal{D}$ of multiplicity $m$,} \\ 0 \quad \hspace{2mm} \text{otherwise}. \end{cases}
\]   
We then set
\[
\sigma \mathbbm{1}_{\mathcal{Z}}={\left.\psi\right|_{\mathcal{D}}}_*\delta,
\]
where ${\left.\psi\right|_{\mathcal{D}}}_*$ denotes the proper pushforward associated with the restriction of $\psi$ to $\mathcal{D}$.
\end{definition} 

It was proven in \cite{VSA} that Definition~\ref{SD} is independent of $\psi$, so that $\sigma \mathbbm{1}_{\mathcal{Z}}$ is indeed well-defined. We now proceed to compute both sides of equation \ref{f3} as given by Theorem~\ref{VT}. We first compute $\varphi_*\mathbbm{1}_{Y}$, which will yield the LHS of identity \ref{f2}. The fibers of $\varphi$ over $B\setminus \Delta$ are all elliptic curves which have vanishing Euler characteristic, the fibers over $\Delta\setminus C$ are nodal cubics which have Euler characteristic equal to 1 and the fibers over $C$ are cuspidal cubics which have Euler characteristic equal to 2. So by Lemma~\ref{l1} we have
\begin{eqnarray*}
\varphi_*\mathbbm{1}_{Y}&=&0\cdot \left(\mathbbm{1}_B-\mathbbm{1}_{\Delta}\right)+1\cdot \left(\mathbbm{1}_{\Delta}-\mathbbm{1}_C\right)+ 2\cdot \mathbbm{1}_C \\
                        &=&\mathbbm{1}_{\Delta}+\mathbbm{1}_C.
\end{eqnarray*}
Theorem~\ref{VT} then yields
\begin{equation}\label{fe}
c_{\text{SM}}\left(\mathbbm{1}_{\Delta}+\mathbbm{1}_C\right)=c_{\text{SM}}\left({\varphi_0}_*\sigma \mathbbm{1}_{\mathscr{Y}}\right),
\end{equation}
thus we now proceed to compute ${\varphi_0}_*\sigma \mathbbm{1}_{\mathscr{Y}}$, which is the more involved part of the calculation. For the computation of $\sigma \mathbbm{1}_{\mathscr{Y}}$, we blowup $\mathbb{P}(\mathscr{E})\times \mathscr{D}$ along the singular locus of $\mathscr{Y}$
\[
\mathscr{Y}_{\text{sing}}:(y=x+h=c=0)\subset \mathbb{P}(\mathscr{E})\times \mathscr{D}.
\]
We set $w=x+h$ for ease of notation, and will work in local coordinates where $z=1$ (as the singularities of $\mathscr{Y}$ are away from $\{z=0\}$) so that we take $(y,w,c;b_1,\ldots, b_r)$ as local coordinates on $\mathbb{P}(\mathscr{E})\times \mathscr{D}$ (the $b_i$ are local coordinates on the base $B$). Our family $\mathscr{Y}$ then has local equation
\[
\mathscr{Y}:\left(y^2=(w-h)^3+(c\eta -3h^2)(w-h)+ch\eta+c^2\chi-2h^3\right)\subset \mathbb{A}^2\times B\times \mathscr{D}.
\]
The blowup map in a representative chart is then given by
\[
(X_1,X_2,X_3;b_1,\ldots, b_r)\mapsto (X_1X_2,X_2,X_2X_3;b_1,\ldots,b_r),
\]
so that the exceptional divisor is given by $X_2^2=0$ and the proper transform of $\mathscr{Y}$ is given by 
\[
\mathscr{Y}_{\text{prop}}:X_1^2-X_2+3h-X_3\eta-X_3^2\chi=0,
\]
which is easily seen to be smooth. As $c=0$ pulls back to $X_2X_3=0$ under the blowup map, the pullback of the central fiber $Y_0$ is obtained by intersecting $\mathscr{Y}_{\text{prop}}$ with $X_2X_3=0$. As such, we see that the pullback of the central fiber $Y_0$ is a divisor with two components of multiplicity 1, one corresponding to intersecting $\mathscr{Y}_{\text{prop}}$ with $X_2=0$ (which yields the proper transform of the central fiber $Y_0$), and one corresponding to intersecting $\mathscr{Y}_{\text{prop}}$ with $X_3=0$. Moreover, the components of this divisor are smooth and intersect transversally (which follows from our smoothness assumption on $O:(h=0)$), and as such is a normal crossing divisor, whose components are explicitly given by the equations
\[
\mathcal{D}_1:X_1^2-3h-X_3\eta-X_3^2\chi=0, \quad \mathcal{D}_2:X_1^2-X_2+3h=0.
\]
The analysis in the other charts of the blowup is similar, and it is then straightforward to verify that our blowup (restricted to $\mathscr{Y}_{\text{prop}}$) satisfies the conditions as given in Definition~\ref{SD}. The function $\delta$ which we must pushforward via the blowup to yield $\sigma \mathbbm{1}_{\mathscr{Y}}$ is then
\[
\delta=\mathbbm{1}_{\mathcal{D}_1}+\mathbbm{1}_{\mathcal{D}_2}-2\mathbbm{1}_{X},
\] 
where we use $X$ to denote $\mathcal{D}_1\cap \mathcal{D}_2$, as it has equation $X_1^2+3h=0$ in $\mathcal{D}_2$ and thus may be naturally identified with the orientifold given in \ref{of} (as pointed out in \cite{SL}). Now we pushforward $\delta$ via the pushforward associated with the blowup composed with ${\varphi_0}_*$ to yield ${\varphi_0}_*\sigma \mathbbm{1}_{\mathscr{Y}}$. For this we view $\mathcal{D}_1$, $\mathcal{D}_2$ and $X$ as fibrations over $B$ and apply Lemma~\ref{l1}. Now $\mathcal{D}_2$ is a $\mathbb{P}^1$-bundle over $B$ and $X$ is a double cover of $B$ ramified over $O:(h=0)$, thus by Lemma~\ref{l1} we have
\[
\mathbbm{1}_{\mathcal{D}_2}\mapsto 2\mathbbm{1}_B, \quad \mathbbm{1}_{X}\mapsto 2\mathbbm{1}_B-\mathbbm{1}_O.
\]
As for $\mathcal{D}_1$, the fibers over $B\setminus D$ are all smooth conics, which degenerate to line pairs intersecting at a single point over $D\setminus S$, and then degenerate further to double lines over $S$. By inclusion-exclusion the Euler characteristic of a line pair intersecting at a single point is 3, thus by Lemma~\ref{l1} we have 
\[
\mathbbm{1}_{\mathcal{D}_1}\mapsto 2\cdot \left(\mathbbm{1}_B-\mathbbm{1}_{D}\right)+3\cdot \left(\mathbbm{1}_D-\mathbbm{1}_{S}\right)+2\cdot \mathbbm{1}_S=2\mathbbm{1}_B+\mathbbm{1}_D-\mathbbm{1}_S.
\] 
We then have
\[
{\varphi_0}_*\sigma \mathbbm{1}_{\mathscr{Y}}=2\mathbbm{1}_O+\mathbbm{1}_D-\mathbbm{1}_S,
\]
which after substituting into the RHS of equation \ref{fe} yields
\[
c_{\text{SM}}(\mathbbm{1}_{\Delta}+\mathbbm{1}_{C})=c_{\text{SM}}\left(2\mathbbm{1}_O+\mathbbm{1}_D-\mathbbm{1}_S\right),
\]
as desired.

\section{Explicit equations for the non-Weiertra{\ss} fibrations, their weak coupling limits and corresponding Chern class identities.}\label{AA}
Let $B$ be a smooth complete complex variety endowed with an ample line bundle $\mathscr{L}\to B$. The non-Weierstra{\ss} fibrations appearing in \cite{AE2}\cite{CCG}\cite{EFY}\cite{EJY} consist of four classes of elliptic fibrations referred to in the physics literature as $E_6$, $E_7$, $D_5$ and $Q_7(\mathscr{L},\mathscr{S})$ fibrations, which we will denote simply by $Q_7$. The explicit equations are given in Table~1.
\begin{table}[hbt]\label{t1}
\begin{center}
\begin{tabular}{|c|c|}
\hline
        \text{Fibration}  & \text{Defining equations}   \\
\hline
$Y_{E_6}$ & $x^3+y^3=a_1xyz+a_2xz^2+a_3yz^2+a_4z^3$ \\
\hline
$Y_{E_7}$ & $y^2=x^4+b_1x^2z^2+b_2xz^3+b_3z^4$ \\
\hline
$Y_{D_5}$ & $x^2-y^2-z(d_1z+d_2w)=w^2-x^2-z(d_3z+d_4x+d_5y)=0$ \\
\hline
$Y_{Q_7}$ & $y(x^2-e_1y^2)+z(e_2y^2+e_3xz+e_4yz+e_5z^2)=0$ \\
\hline
\end{tabular}
\end{center}
\caption{Defining equations for non-Weierstra{\ss} fibrations.}
\end{table}

The $E_6$ and $Q_7$ fibrations are both hypersurfaces in a $\mathbb{P}^2$-bundle over $B$, the $E_7$ fibration is a hypersurface in a $\mathbb{P}(1,1,2)$-bundle and the $D_5$ fibration is a complete intersection in a $\mathbb{P}^3$-bundle. The coefficients $a_i$, $b_j$ and $c_k$ are all general sections of appropriate tensor powers of $\mathscr{L}$ so that each of the defining equations for the fibration determines a well defined section of a line bundle on the ambient projective bundle in which they reside. The coefficients $e_l$ are all general sections of line bundles of the form $\mathscr{L}^m\otimes \mathscr{S}^n$ where $\mathscr{S}\to B$ is another line bundle on $B$. In each of the $E_6$, $D_5$ and $Q_7$ cases, a single weak coupling limit was constructed for each fibration in \cite{AE2}\cite{EFY}\cite{EJY} respectively by setting
\[
\begin{cases}
E_6:a_1=6\nu,\hspace{3mm}a_2=9\nu^2+3h,\hspace{3mm}a_3=9\nu^2+3h+c\phi,\hspace{3mm}a_4=2\nu(5\nu^2+3h)+c(\chi+\nu \phi) \\
D_5:d_1=c\alpha,\hspace{3mm}d_2=c\eta,\hspace{3mm}d_3=h,\hspace{3mm}d_4=\psi_1+\psi_2,\hspace{3mm}d_5=\psi_1-\psi_2 \\
Q_7:e_1=\beta,\hspace{3mm}e_2=2\vartheta,\hspace{3mm}e_3=c^2\rho,\hspace{3mm}e_4=h,\hspace{3mm}e_5=c\iota, \\
\end{cases}
\]  
where the Greek letters all denote general sections of line bundles of the form $\mathscr{L}^p$ (or $\mathscr{L}^m\otimes \mathscr{S}^n$ in the $Q_7$ case) and $c$ is a deformation parameter which in each case gives rise to a family over a disk $\mathscr{Y}\to \mathscr{D}$. For the $E_7$ case two weak coupling limits were constructed in \cite{AE2} corresponding to setting
\[
b_1=-2h, \quad b_2=c\delta, \quad b_3=h^2+c\gamma
\]
for the first limit, and setting
\[
b_1=h-6\epsilon^2, \quad b_2=2\epsilon(h-4\epsilon^2)+2c\zeta, \quad b_3=\epsilon^2(h-3\epsilon^2)+2c\epsilon \zeta+c^2\tau
\]
for the second. The corresponding families are given in Table~2.

\begin{table}[hbt]\label{t2}
\begin{center}
\begin{tabular}{|c|c|}
\hline
        \text{Family}  & \text{Defining equations}   \\
\hline
$\mathscr{Y}_{E_6}$ & $x^3+y^3=6\nu xy+(9\nu^2+3h)x+(9\nu^2+3h+c\phi)y+10\nu^6\nu h+c(\chi+\nu \phi)$ \\
\hline
$\mathscr{Y}_{E_7}$ & $y^2=(x^2-h)^2+c(\delta x+\gamma)$ \\
\hline
$\mathscr{Y}'_{E_7}$ & $y^2=x^4+(h-6\epsilon^2)x^2+2\epsilon(h-4\epsilon^2)x+\epsilon^2(h-3\epsilon^2)+c(2\epsilon \zeta +c\tau)$ \\
\hline
$\mathscr{Y}_{D_5}$ & $x^2-y^2-c(\alpha-\eta w)=w^2-x^2-(h+(\psi_1+\psi_2)x+(\psi_1-\psi_2)y)=0$ \\
\hline
$\mathscr{Y}_{Q_7}$ & $y(x^2-\beta y^2+2\vartheta y+h)+c(\iota+c\rho x)=0$ \\
\hline
\end{tabular}
\end{center}
\caption{The families corresponding to weak coupling limits of non-Weierstra{\ss} fibrations.}
\end{table}

The singular locus of each family is either in codimension three or codimension four in the ambient space of the family $\mathbb{A}^2\times B\times \mathscr{D}$ ($\mathbb{A}^3\times B\times \mathscr{D}$ for the $D_5$ case). In each case a resolution of singularities of the total space of the family may be achieved by a single blowup, the centers of which we list in Table~3 (along with the singular locus of each family).

\begin{table}[hbt]\label{t3}
\begin{center}
\begin{tabular}{|c|c|c|}
\hline
        \text{Family} &\text{Singular locus}  & \text{Center of blowup}   \\
\hline
$\mathscr{Y}_{E_6}$ & $x+y+2\nu=h-(y+\nu)^2=\chi+\phi(y+\nu)=c=0$ & $x+y+2\nu=c=0$ \\
\hline
$\mathscr{Y}_{E_7}$ & $y=x^2-h=\delta x+\gamma=c=0$ & $y-x^2+h=c=0$ \\
\hline
$\mathscr{Y}'_{E_7}$ & $y=x+\epsilon=c=0$ & $y=x+\epsilon=c=0$ \\
\hline
$\mathscr{Y}_{D_5}$ & $x=y=\alpha-\eta w=c$ & $x-y=c=0$ \\
\hline
$\mathscr{Y}_{Q_7}$ & $y=x^2+h=\iota=c=0$ & $y=c=0$ \\
\hline
\end{tabular}
\end{center}
\caption{Equations in $\mathbb{A}^2\times B\times \mathscr{D}$ ($\mathbb{A}^3\times B\times \mathscr{D}$ for the $D_5$ case) for the singular loci of the families and for the centers of the blowups needed for a resolution of singularities for the corresponding family.}
\end{table}

In each case, the pullback of the central fiber of the original family pulls back to a normal crossing divisor, whose smooth components are all conic fibrations over the base. The discriminant of each of these conic fibrations coincide with the components of the limiting discriminant of the corresponding family other than the orientifold plane $O:(h=0)$, which are given in Table~4 (note that $D_2$, $D_3$, $D_6$, $D_9$ and $D_{11}$ are all singular).

\begin{table}[hbt]\label{t4}
\begin{center}
\begin{tabular}{|c|c|}
\hline
        \text{Limiting discriminant}  & \text{Defining equation}   \\
\hline
$\Delta_{E_6}=O\cup D_1 \cup D_2$ & $h^2(h+3\nu^2)(\chi^2-h\phi^2)=0$ \\
\hline
$\Delta_{E_7}=O\cup D_3$ & $h^2(\gamma^2-\delta^2h)=0$ \\
\hline
$\Delta'_{E_7}=O\cup D_4 \cup D_6$ & $h^2(h+4\epsilon^2)(\zeta^2-h\tau)=0$ \\
\hline
$\Delta_{D_5}=O\cup D_7 \cup D_8 \cup D_9$ & $h^2(h-\psi_1^2)(h-\psi_2^2)(h\eta^2-\alpha^2)=0$ \\
\hline
$\Delta_{Q_7}=O\cup D_{10} \cup D_{11}$ & $h^2\iota^2(\vartheta^2+h\beta)=0$ \\
\hline
\end{tabular}
\end{center}
\caption{The limiting discriminants for each family.}
\end{table}

The presumed tadpole relations predicted by physical arguments are then obtained by setting the Euler characteristic of the elliptic fibration to twice the Euler characteristic of the orientifold plane $O:(h=0)$ plus the Euler characteristics of the other components of the limiting discriminant (modulo a correction factor accounting for possible singularities of the components). For example, in the $E_6$ case the limiting discriminant is given by
\[
\Delta_{E_6}:(h^2(h+3\nu^2)(\chi^2-h\phi^2)=0)\subset B 
\]  
Thus the tadpole relation is given by 
\begin{equation}\label{E6TP2}
\chi(Y_{E_6})=2\chi(O)+\chi(D_1)+\chi(D_2)-\chi(S_2)
\end{equation}
where $D_1$, $D_2$ and $S_2$ are given by
\[
D_1:(h+3\nu^2=0), \quad D_2:(\chi^2-h\phi^2=0), \quad S_2:(\chi=h=\phi=0),
\]
so that $S_2$ is the pinch locus of $D_2$.

As in the Weierstra{\ss} case, the tadpole relation (\ref{E6TP2}) predicted  by physics is in fact the degree zero component of equation (\ref{mt}) of Theorem~\ref{VT} for the case $\mathscr{Y}=\mathscr{Y}_{E_6}$, namely
\begin{equation}\label{E6TP}
c_{\text{SM}}\left(\varphi_*\mathbbm{1}_{Y_{E_6}}\right)=c_{\text{SM}}\left({\varphi_0}_*\sigma \mathbbm{1}_{\mathscr{Y}_{E_6}}\right).
\end{equation}
In this case we have  
\[
{\varphi_0}_*\sigma \mathbbm{1}_{\mathscr{Y}_{E_6}}=2\mathbbm{1}_O+\mathbbm{1}_{D_1}+\mathbbm{1}_{D_2}-\mathbbm{1}_{S_2},
\]
so that the Chern class identity corresponding to the $E_6$ tadpole relation is 
\[
c_{\text{SM}}\left(\varphi_*\mathbbm{1}_{Y_{E_6}}\right)=c_{\text{SM}}\left(2\mathbbm{1}_O+\mathbbm{1}_{D_1}+\mathbbm{1}_{D_2}-\mathbbm{1}_{S_2}\right).
\]

\begin{remark}
The physicist's tadpole relation in the $E_6$ case looks a bit different from our equation (\ref{E6TP2}), namely,
\[
\chi(Y_{E_6})=2\chi(O)+\chi\left({D_1}_+\right)+\chi\left({D_2}_+\right),
\]
where ${D_i}_+$ denotes one of the smooth components of the pullback of $D_i$ to the orientifold
\[
X:(\xi^2=h)\subset \mathscr{L}.
\]
In particular, $D_1:(h+3\nu^2=0)$ pulls back to 
\[
{D_1}_+\cup {D_1}_-:\left((\xi+i\sqrt{3}\nu)(\xi-i\sqrt{3}\nu)=0\right)\subset \mathscr{L},
\]
and $D_2:(\chi^2-h\phi^2)$ pulls back to 
\[
{D_2}_+\cup {D_2}_-:\left((\chi+\xi \phi)(\chi-\xi \phi)=0\right)\subset \mathscr{L}.
\]
Both pullbacks are normal crossing divisors with two components, which physicists refer to as `brane-image-brane pairs'. But since our identities are at the level of constructible functions on the base $B$ we do not pullback the components of the limiting discriminant to the orientifold, and thus our identities look slightly different in the cases where a singular component of the limiting discriminant pulls back to a brane-image-brane pair. In particular, this phenomenon occurs in the $D_5$ case and the (unprimed) $E_7$ case as well.  
\end{remark}

We then conclude by listing the Chern class identities corresponding to the families of the other non-Weierstra{\ss} fibrations in Table~5, which are all just reflections of Theorem~\ref{VT} for the corresponding family\footnote{In each case, $S_j$ is the intersection of the singular locus of $D_j$ with the orientifold plane.}.

\begin{table}[hbt]\label{t5}
\begin{center}
\begin{tabular}{|c|c|}
\hline
        \text{Family}  & \text{Chern class identity following from Theorem~\ref{VT}} \\
\hline
$\mathscr{Y}_{E_7}$ & $c_{\text{SM}}\left(\varphi_*\mathbbm{1}_{Y_{E_7}}\right)=c_{\text{SM}}\left(2\mathbbm{1}_O+\mathbbm{1}_{D_3}-\mathbbm{1}_{S_3}\right)$ \\
\hline
$\mathscr{Y}'_{E_7}$ & $c_{\text{SM}}\left(\varphi_*\mathbbm{1}_{Y_{E_7}}\right)=c_{\text{SM}}\left(2\mathbbm{1}_O+\mathbbm{1}_{D_4}+\mathbbm{1}_{D_6}-\mathbbm{1}_{S_6}\right)$ \\
\hline
$\mathscr{Y}_{D_5}$ & $c_{\text{SM}}\left(\varphi_*\mathbbm{1}_{Y_{D_5}}\right)=c_{\text{SM}}\left(2\mathbbm{1}_O+\mathbbm{1}_{D_7}+\mathbbm{1}_{D_8}+\mathbbm{1}_{D_9}-\mathbbm{1}_{S_9}\right)$ \\
\hline
$\mathscr{Y}_{Q_7}$ & $c_{\text{SM}}\left(\varphi_*\mathbbm{1}_{Y_{Q_7}}\right)=c_{\text{SM}}\left(2\mathbbm{1}_O+2\mathbbm{1}_{D_{10}}-\mathbbm{1}_{S_{10}}+\mathbbm{1}_{D_{11}}-\mathbbm{1}_{S_{11}}\right)$ \\
\hline
\end{tabular}
\end{center}
\caption{Chern class identities following from Theorem~\ref{VT} applied to the corresponding family. In each case, integrating both sides of the Chern class identity yields the presumed tadpole relation predicted by F-theory/type IIB duality.}
\end{table}

\bibliographystyle{plain}
\bibliography{TR2}

\end{document}